\documentclass[12pt]{article}
\usepackage{amsmath}
\usepackage{amsfonts}
\usepackage{amssymb}
\usepackage{amsthm}
\usepackage{stackrel}

\usepackage{color}

\usepackage{figsize}

\usepackage{graphics}

\newtheorem{thm}{Theorem}[section]

\newcommand{\R}{{\rm I}\kern-0.18em{\rm R}}
\newcommand{\1}{{\rm 1}\kern-0.25em{\rm I}}
\newcommand{\E}{{\rm I}\kern-0.18em{\rm E}}
\newcommand{\p}{{\rm I}\kern-0.18em{\rm P}}
\makeatletter

\title{Modified Stieltjes Transform and Generalized Convolutions}
\author{Lev B Klebanov\footnote{Department of Probability and Mathematical Statistics, MFF, Charles University, Prague, Czech Republic}, Rasool Roozegar\footnote{Department of Statistics, Yazd University, P.O. Box 89195-741, Yazd, Iran}}
\date{}
\begin{document}
	
\maketitle

\begin{abstract}
Classical Stieltjes Transform is modified in a way to generalize both Stieltjes and Fourier transforms. This transform allows to introduce new classes of commutative and non-commutative generalized convolutions.

\end{abstract}	

\section{Introduction}\label{s1}
\setcounter{equation}{0} 

Let us begin with a definitions of classical and generalized Stieltjes transforms. Although, usual these are transforms given on a set of functions, we will consider more convenient for us case of probability measures or for cumulative distribution functions. Namely, let $\mu$ be a probability measure of Borel subsets of real line $\R^1$. Its Stieltjes transform is defined as
\[S(z)=S(z;\mu) = \int_{-\infty}^{\infty}\frac{d \mu(x)}{x-z}, \]
where ${\tt Im} (z) \neq 0$. Surely, the integral converges in this case. Generalized Stieltjes transform is represented by
\[ S_{\gamma}(z)=S_{\gamma}(z;\mu)= \int_{-\infty}^{\infty}\frac{d \mu(x)}{(x-z)^\gamma}\]
for real $\gamma >0$. A modification of generalized Stieltjes transform was proposed in \cite{RR}. Now we prefer to change this modification, and define the following form of transform:
\begin{equation}\label{eq1}
R_{\gamma}(u) = R_{\gamma}(u;\mu) =\int_{-\infty}^{\infty}\frac{d \mu(x)}{(1-iux)^\gamma}.
\end{equation}
Connection to the generalized Stieltjes transform is obvious. It is convenient for us to use this transform for real values of $u$. It is clear that the limit
\begin{equation}\label{eq2}
\lim_{\gamma \to \infty}R_{\gamma}(u/\gamma) = \int_{-\infty}^{\infty} \exp\{iux\}d \mu(x),
\end{equation}
represents Fourier transform (characteristic function) of the measure $\mu$. The uniqueness of a measure recovering from its modified Stieltjes transform follows from corresponding result for generalized Stieltjes transform. 

Relation (\ref{eq2}) gives us limit behavior of modified Stieltjes transform as $\gamma \to infty$. Another possibility ($\gamma \to 0$) without any normalization gives trivial limit equals to 1. However, more proper approach is to calculate the limit $(R_{\gamma}(u)-1)/\gamma$ as $\gamma \to 0$. It is easy to see, that
\begin{equation}\label{eq2a}
\lim_{\gamma \to 0}(R{\gamma}(u)-1)/\gamma = \int_{-\infty}^{\infty}\log \frac{1}{1-iux} d\mu(x).
\end{equation}

If the measure $\mu$ has compact support, it is possible to write series expansion for modified Stieltjes transform: 
\[R_{\gamma}(u)=\int_{-\infty}^{\infty} \frac{d \mu(x)}{(1-i u x)^{\gamma}}=\sum_{k=0}^{\infty}(-1)^k i^k \binom{-\gamma}{k} \kappa_k(\mu) x^k ,\]
where $\kappa_k(\mu) = \int_{-\infty}^{\infty}x^k d\mu (x)$ is $k$th moment of the measure $\mu$. 

Modified Stieltjes transform may be interpreted in terms of characteristic functions. Namely, let us consider gamma distribution with probability density function
\begin{equation}\label{eqG1}
p(x)= \frac{1}{|\lambda|^{\gamma}\Gamma (\gamma)}x^{\gamma -1}\exp(-x/\lambda),
\end{equation}
for $x*\lambda >0$, and zero in other cases. Note, that this distribution is ordinary gamma distribution for positive $\lambda$, and its "mirror reflection" on negative semi-axes for negative $\lambda$. Let us now consider $\lambda$ as random variable with cumulative distribution function $\mu$. In this case, (\ref{eq1}) gives characteristic function of gamma distribution with such random parameter:
\begin{equation}\label{eqG2}
f(t)=\int_{-\infty}^{\infty}\frac{d \mu (\lambda)}{(1-it\lambda)^{\gamma}}.
\end{equation}

\section{A family of commutative generalized convolutions}\label{s2}
\setcounter{equation}{0} 

Using modifies Stieltjes transform we can introduce a family of commutative generalized convolutions. Main idea for this is the following. Let $\mu_1$ and $\mu_2$ be two probabilities. Take positive $\gamma$ and consider product of modified Stieltjes transforms of these measures $R_{\gamma}(u,\mu_1)R_{\gamma}(u,\mu_2)$. We would like to represent this product as a modified Stieltjes transform of a measure. Typically, the product is not modified Stieltjes transform with the same index $\gamma$. However, it can be represented as modified Stieltjes transform with index $\rho>\gamma$ of a measure $\nu$, which is called  generalized (more precisely $"(\gamma,\rho)"$) convolution of the measures $\mu_1$ and $\mu_2$. Let us mention that the indexes $\rho$ and $\gamma$ are not arbitrary, however, there are infinitely many suitable pairs of indexes. Clearly, the measure $\nu$, if exists, depends on $\mu_1$, $\mu_2$, and on indexes $\gamma$, $\rho$. 

Unfortunately, we cannot describe all pairs $\gamma,\rho$ for which corresponding generalized convolution $\nu$ of measures $\mu_1$ and $\mu_2$ exists. However, we shall show, the pairs of the form $n,2n$ (where $n$ is positive, but not necessarily integer number) possess this property.

\begin{thm}\label{th1}
Let $\mu_1$, $\mu_2$ be two probability measures on $\sigma$-field Borel subsets of real line. For arbitrary real $n>0$ there exists $"(n,2n)"$ convolution $\nu$ of $\mu_1$ and $\mu_2$. In other words, for real $n>0$ and measures $\mu_1$ and $\mu_2$ there exists a measure $\nu$ such that
\begin{equation}\label{eq3}
R_{2n}(u;\nu)=R_{n}(u;\mu_1) R_{n}(u;\mu_2).
\end{equation}
\end{thm}
\begin{proof}
Because convex combination of probability measures is a probability measure again, and each probability on real line can be considered as a limit of sequence of measures concentrated in finite number of points each, it is sufficient to prove the statement for Dirac $\delta$-measures only.

Suppose now that the measures $\mu_1$ and $\mu_2$ are concentrated in points $a$ and $b$ correspondingly. We have to prove that there is a measure $\nu$ depending on $a$, $b$ and $n$ such that
\begin{equation}\label{eq3a}
\int_{-\infty}^{\infty}\frac{d \nu(x)}{(1-iux)^{2n}}=\frac{1}{(1-iua)^{n}}\cdot\frac{1}{(1-iub)^{n}}.
\end{equation}
Of course, it is enough to find the measure $\nu$ with compact support. Therefore, we must have
\[\kappa_m =\]
\begin{equation}\label{eq4}
\sum_{k=0}^{m}\frac{n(n+1)\cdots (n +k-1)}{k!}\cdot \frac{n (n+1)\cdots (n+m-k-1)}{(m-k)!} a^k b^{m-k}/\binom{-2n}{m},
\end{equation}
where $\kappa_m = \kappa_m(\nu)$ is $m$th moment of $\nu$. It remains to show that the left hand side of (\ref{eq4}) really defines for $m=0,1, \ldots$ moments of a distribution.

Let us denote $\lambda =a/b$ and suppose that $|\lambda|<1$ (the case  $|\lambda|=1$ may be obtained as a limit case). Then $\kappa_m$ can be rewrite in the form
\[ \kappa_m = (-1)^m b^m \sum_{k=0}^{m}\binom{m}{k}\frac{(n)_k (n)_{m-k}}{(2n)_m}\lambda^k, \]
where $(s)_j=s\cdots (s+j-1)$ is Pochhammer symbol. Simple calculations allows us to obtain from previous equality that
\begin{equation}\label{eq5}
\kappa_m=\frac{b^m (n)_{m}\; {}_2F_1 (-m,n,1-m-n,a/b)}{(2n)_m}.
\end{equation}

Let us consider a random variable $X$ having Beta distribution with equal parameters $n$ and $n$, that is with probability density function
\[ p_X(x)=(1-x)^{n-1}x^{n-1}2^{2n-1}\Gamma(n+1/2)/(\sqrt{\pi}\, \Gamma(n)),\]
for $x \in (0,1)$, and zero for $x \notin (0,1)$. It is not difficul to calculate that 
\[\E \Bigl( aX+b(1-X)\Bigr)^m = b^m \; {}_2F_1 (-m,n,2n,1-a/b), \]
which coincide with (\ref{eq5}) for non-negative integer $m$ and real $n>0$.
\end{proof}

Theorem \ref{th1} allows us to define a family of depending on $n$ generalized convolutions $\nu=\mu_1\star_{n} \mu_2$, which is equivalent to the relation (\ref{eq3}). Obviously, this operation is commutative. However, it is not associative, which can be easily verified by comparing the convolutions $(\delta_1\star_{n}\delta_2)\star_{n}\delta_3$ and $\delta_1\star_{n}(\delta_2\star_{n}\delta_3)$, where $\delta_a$ denotes Dirac measure at point $a$. It is easy to verify that $\mu_1\star_{n} \mu_2(2 A) \stackrel[n \to \infty]{}{\longrightarrow} \mu_1*\mu_2(A)$, where $*$ denotes ordinary convolution of measures. We have $2 A$ in the left-hand-side because $\E X=1/2$. This generalized convolution may be written through independent random variables $U$ and $V$ in the form
\[ W=U X+V (1-X),\]
where $X$ is random variable independent of $(U,V)$ and having Beta distribution with parameters $(n,n)$, and the distribution of $W$ is exactly generalized convolution of distributions of $U$ and $V$.

In view of non-associativity of $\star_{n}$-convolution it does not coincide with K. Urbanik generalized convolution (see, \cite{Urb}). At the same time, it non-associativity shows that the expression $\mu_1\star_{n}\mu_2\star_{n}\mu_3$ has no sense. However, one can define this 3-arguments operation by use stochastic linear combinations, that is linear forms of random variables with random coefficients. Now we define such $k$-arguments operation. Namely, let $U_1, \ldots , U_k$ be independent random variables, and $X_1, \ldots , X_{n-1}$ be a random vector having Dirichlet distribution with parameters $(a_1, \ldots , a_k)=(n, \ldots ,n)$. Define
\begin{equation}\label{eq6}
W= X_1 U_1 + \ldots +X_{k-1}U_{k-1}+\bigl(1-\sum_{j=1}^{k-1}\bigr)\,U_k
\end{equation}
The map from vector $U$ of marginal distributions of $(U_1, \ldots ,U_k)$ to the distribution of random variable $W$ call $k$-tuple generalized convolution of the components of $U$. Clearly, this operation is symmetric with respect to permutations of coordinates of the vector $U$.

\section{Connected family of non-commutative generalized convolutions}\label{s3}
\setcounter{equation}{0} 

Let now $U_1, \ldots , U_k$ be independent random variables, and $X_1, \ldots , X_{n-1}$ be a random vector having Dirichlet distribution with parameters $(a_1, \ldots , a_k)$, possible different from each other. Using the relation (\ref{eq6}) define random variable $W$. Its distribution will be called non-commutative generalized convolution of marginal distributions of the vector $U$. In particular case of $k=2$ we obtain non-commutative variant of two-tuple generalized convolution, which represents more general case of (\ref{eq1}).

Let us give a property of this generalized convolution.
To do so, let us define $\tilde{b}eta_{A,B}$ distribution over interval $(A,B)$ by its probability density function
\[p_{\alpha,\beta}(x) =\begin{cases}
\frac{1}{B(\alpha,\beta)(B-A)^{\alpha + \beta-1}}(x-A)^{\alpha-1}(B-x)^{\beta -1}, & \text{if} \; A<x<B,\\ 0 & \text{otherwise},
\end{cases} \]
for positive $\alpha, \; \beta$. Here $B(\alpha , \beta)$ is beta function.

\begin{thm}\label{th2}
Let $W_1,W_2$ be two independent identical distributed random variables having $\tilde{b}eta_{A,B}(n,n)$ distribution, and $\mu_1,\mu_2$ be corresponding probability distributions. Then the measure $\nu=\mu_1 \star_n \mu_2$ corresponds to $\tilde{b}eta_{A,B}(2n,2n)$ distribution.
\end{thm}
\begin{proof}
From the proof of Theorem \ref{th1} that $W_j \stackrel{d}{=}A X_j+B(1-X_j)$, where $X_1,X_2$ are independent identically distributed random variables having Beta(n,n) distribution. The rest of the proof is just simple calculation.
\end{proof}
The property given by Theorem \ref{th2} is very similar to classical stability definition. 

\begin{thm}\label{th3}
Let $U_j$, $j=1, \ldots ,k$ be independent random variables having $\tilde{b}eta$ distribution with parameters $\alpha_j=r_j+1/2$, $\beta_j=r_j+1/2$. Let $X_1, \ldots , X_{k-1}$ be a random vector having Dirichlet distribution with parameters $(r_1, \ldots , r_k)$. Then random variable
\[ W= X_1 U_1 + \ldots +X_{k-1}U_{k-1}+\bigl(1-\sum_{j=1}^{k-1}X_{j}\bigr)\,U_k \]
has $\tilde{b}eta$ distribution with parameters $\Bigl(\sum_{j=1}^{k}r_j+1/2, \;\sum_{j=1}^{k}r_j+1/2\Bigr)$.
\end{thm}
\begin{proof}
It is sufficient to calculate modified Stieltjes transform of the distribution of $W$ using some properties of Gauss-hypergeometric function.
\end{proof}
This property is also similar to classical stability property, but for the case of k-tuple operation.

\section{Acknowledgment}
The work was partially supported by Grant GACR 16-03708S.

\end{document}